\documentclass[a4paper,12pt]{article}

\usepackage[T1]{fontenc}
\usepackage[utf8]{inputenc} 
\usepackage[english]{babel} 
\usepackage{caption} 
\usepackage{amsthm}
\usepackage{mathenv} 
\usepackage{fancyhdr} 
\usepackage{graphicx} 
\usepackage{enumitem}
\usepackage{setspace} 
\usepackage{tikz}
\usepackage{tkz-tab}
\usepackage{empheq}
\usepackage{amsmath,amsfonts,amssymb}
\usepackage{version} 
\usepackage{mathtools}						
\mathtoolsset{showonlyrefs,showmanualtags}	
\usepackage{xcolor}
\usepackage[colorlinks,linkcolor=blue,citecolor=blue,urlcolor=blue]{hyperref} 

\usepackage{pgfplots}
\pgfplotsset{compat=1.17}

\newtheorem{theorem}{Theorem}[section]
\newtheorem{lem}[theorem]{Lemma}
\newtheorem{prop}[theorem]{Proposition}
\newtheorem{coro}[theorem]{Corollary}
\newtheorem{rem}{Remark}[section]
\newtheorem{assum}{Assumption}

\theoremstyle{definition}
\newtheorem{problem}{Problem}

\newcommand\R{\mathbb{R}}
\newcommand\U{\mathcal{U}}

\newcommand\sca[2]{\langle #1, #2 \rangle}

\author{Clara Leparoux \footnotemark[1] \footnotemark[2]\and Bruno H\'{e}riss\'{e} \thanks{ Clara Leparoux ({clara.leparoux@ensta-paris.fr}) and Bruno H\'{e}riss\'{e} ({bruno.herisse@onera.fr}) are with DTIS, ONERA, Université Paris-Saclay, F-91123 Palaiseau, France.
        } \and Fr\'{e}d\'{e}ric Jean \thanks{ Clara Leparoux and Fr\'{e}d\'{e}ric Jean ({frederic.jean@ensta-paris.fr}) are with UMA, ENSTA Paris,
        Institut Polytechnique de Paris, 91120 Palaiseau, France.
        }
        }
\date{}
\begin{document}

\title{\LARGE \bf
Structure of optimal control for planetary landing with control and state constraints}

\maketitle



\begin{abstract}
This paper studies a vertical powered descent problem in the context of planetary landing, considering glide-slope and thrust pointing constraints and minimizing any final cost. In a first time, it proves the Max-Min-Max or Max-Singular-Max form of the optimal control using the Pontryagin Maximum Principle, and it extends this result to a problem formulation considering the effect of an atmosphere. It also shows that the singular structure does not appear in generic cases. 
In a second time, it theoretically analyzes the optimal trajectory for a more specific problem formulation to show that there can be at most one contact or boundary interval with the state constraint on each Max or Min arc.
\end{abstract}

\section{Introduction}
Interest for vertical landing problems has increased these recent years with the growing use of reusable launchers and as space exploration missions are requiring ever more landing precision. Vertical landing consists in two phases$:$ first an entry phase, that set the vehicle in appropriate position and velocity conditions above the target to ensure landing feasibility, and then a phase of powered descent. Prior to the powered descent, the reference trajectory has to be recalculated in order to correct dispersions accumulated during the entry phase (\cite{Blackmore2016}), or handle an update of the landing site. This motion planning problem is usually treated in an optimal control framework, allowing to minimize a quantity such as fuel consumption, flight time, or landing errors to ensure soft landing. 
\par However, solving an optimal control problem efficiently and accurately remains a challenge; moreover, analytical solutions, computed thanks to the Pontryagin Maximum Principle, are only known for simplified models with few constraints, as in \cite{Acikmese2007} and \cite{Lu2018}. This paper is concerned with the following dynamic model, expressed in an inertial frame $(e_x, e_y, e_z)$,
\begin{equation}\label{eq:dyn}
\begin{cases}
\dot{r} &= v, \\
\dot{v} &= \frac{T}{m} u - g, \\
\dot{m} &= -q \Vert u \Vert,\\
\end{cases}
\end{equation}
\noindent
where $r(t)= (x(t), y(t), z(t)) \in \R^3$ is the vehicle position, $v(t) \in \R^3$ is its velocity, $m(t)>0$ is its mass, $q$ the maximal mass flow rate of the engine, $T >0$ the maximal thrust and $g = (0,0,g_0)$ with $g_0$ the gravitational constant. The thrust is controlled by the vector $u(t) \in \R^3$, where $||u|| \leq 1$ is the engine throttle. Previous theoretical and numerical studies tend to show that the structure of the optimal control for dynamics \eqref{eq:dyn} generally consists of a finite succession of arcs on which the control norm is constant. Indeed, it is shown in \cite{Meditch1964} that the optimal control of the one dimension fuel-optimal problem with a bounded control follows an Off-Bang structure, i.e. it has a period of off thrust followed by full thrust until touchdown. Then, the Bang-Off-Bang structure, called Max-Min-Max when the thrust is not allowed to go to zero, has been found to be optimal for numerous variants of the problem, such as the two-dimensional problem studied in \cite{LEITMANN1959}, the problem with specified initial and final thrust studied in \cite{Topcu2005}, or with throttle and thrust angle control in \cite{Topcu2007}. Recently, \cite{Gazzola2021} and \cite{Menou2021} showed that the Max-Min-Max structure is the solution of the one dimension problem, minimizing the final time for the first one and considering the effect of an atmosphere on the thrust for the latter. However, there lack theoretical studies relating more complex formulations of the landing problem, for instance considering realistic technical and safety constraints. Among the existing numerical methods, a noteworthy one is presented in \cite{Acikmese2007, blackmore2010, acikmese2013}, that succeeds to solve efficiently the problem including constraints on the thrust direction and the launcher position, thanks to a convexification method. The simulations carried out in these studies also reveal a Max-Min-Max form of control.\par
The purpose of this paper is to analyze a vertical landing problem considering relevant control and state constraints to highlight the rigidity of the solution. Indeed, we show that the structure of the control remains identical after changing the cost or adding constraints, which suggests that it is not specific to the problem formulation (we can even modify the dynamic model to take into account the effect of an atmosphere). We will study at first the problem with dynamics \eqref{eq:dyn} and we will consider, in addition to the bound on the control norm, a thrust pointing constraint limiting the amplitude of the control direction as well as the launcher orientation, both for safety reasons and to model actuator limitations. Then, we also take into account a glide-slope constraint, forcing the launcher position to stay inside a cone centered on the target, to ensure that the vehicle remains at a safe altitude and to guarantee sensor operability. The cost of the optimal control problem considered is expressed as a final cost, which embraces the most common applications such as maximizing the final mass or minimizing the final time. In this framework (defined later as Problem \ref{pb:completePb}), we can show that the optimal control has either a Max-Min-Max form or a Max-Singular-Max form as stated in the following theorem
\begin{theorem}\label{th:BOB}
Consider an optimal trajectory on $[0,t_f]$. Then, the control $u(t)$ is in the Max-Min-Max or the Max-Singular-Max  form, i.e. there exists $t_1$ and $t_2$ with $0 \leq t_1 \leq t_2 \leq t_f$ such that 
\[\|u\|(t) =
\begin{cases}
& u_{max} \; \normalfont\text{if} \; t \in [ 0, t_1)\cup (t_2, t_f], \\
& u_{min} \;  \normalfont\text{or singular if} \; t \in [t_1, t_2].\\
\end{cases}
\]
\end{theorem}
\begin{rem}
When $t_1 = 0$ or $t_2 = t_f$, the Max-Min-Max or Max-Singular-Max form degenerates in a Max, Min, Singular, Max-Min, Max-Singular, Min-Max or Singular-Max form.
\end{rem}
This result has been presented without proof in \cite{CDC2022}; here we will detail the complete proof of Theorem \ref{th:BOB}. Moreover, without the pointing constraint, we show that this result stays true when considering the effect of the atmosphere as a constant pressure. Finally, we analyze a particular case, in which the mass is assumed constant and the glide-slope constraint reduced to a positive altitude constraint, to show that in that configuration there are at most three contacts with the state constraint. Besides, the numerical results available make think that this outcome is likely to be generalized. 
\par This paper is organized as follows. In Section II, the full optimal control problem with the constraints considered is presented, and the optimality conditions given by application of PMP are detailed. In Section III, the structure of the optimal control is stated and the proof is exposed. Section IV studies a particular case of the landing problem to show that the number of contacts with the state constraint is limited. Finally, Section V provides numerical results.

\section{Problem statement} \label{sec:2}
\subsection{Formulation}
The study will concern vehicles having dynamics in the form of \eqref{eq:dyn}. The initial mass of the vehicle will be denoted by $m_0$ and its empty mass by $m_e$. The state is denoted by $X=(r,v,m) \in \R^3 \times \R^3 \times \R$. It will consider the following constraints:
\begin{itemize}
    \item bounds on the control norm, 
    \[ u_{min} \leq \| u(t) \| \leq u_{max}, \]
    \item a pointing constraint, formulated as \[ \sca{e_z}{u} \geq \Vert u \Vert \cos(\theta),\ \text {with} \ \theta \in [0, \frac{\pi}{2}),\]
    \begin{center}
    \begin{figure}
        \begin{minipage}[b]{0.5\linewidth}
        \centering
    \begin{tikzpicture}[scale=0.55]
\def \tf{5}
\def \hmax{4}
\def \thet{60}

\draw[very thick, ->] (0,-\hmax)  -- (0,\hmax) node[left, pos=1.1]{\large $\frac{p_{v_z}}{\|p_v\|}$};
\draw[line width=0.8mm, ->] (0,0) -- (0,0.6*\hmax)node[right]{$e_z$}  -- (0,0.8*\hmax) ;

\draw[very thick, ->] (-\hmax,0) -- (\hmax,0) node[below]{\large $\frac{p_{v_x}}{\|p_v\|}$};
 \fill[gray, opacity = 0.2] (0,0) circle (0.8*\hmax);
 \draw (0,0) circle (0.8*\hmax);
 
\draw (0,0) --({0.9*\hmax*sin(\thet)},{0.9*\hmax*cos(\thet)}) arc ({90-\thet}:{90+\thet}:.9*\hmax) --(0,0) -- cycle ;
\fill[blue, opacity = 0.2] (0,0) --({0.9*\hmax*sin(\thet)},{0.9*\hmax*cos(\thet)}) arc ({90-\thet}:{90+\thet}:.9*\hmax) --(0,0) -- cycle ;

\draw[dashed,line width=0.4mm,  ->] ++(0, \hmax*0.4) arc ({90}:{90-\thet}:.4*\hmax)node[pos=0.3, below]{$\theta$} ; 
\draw[red, line width=0.8mm, ->] (0,0) --  ({-0.8*\hmax*sin(30)},{0.8*\hmax*cos(30)})node[pos = 0.8, right]{$1$}  ;

\draw[black!50!green, line width=0.8mm, ->] (0,0) --  ({0.8*\hmax*cos(-30)},{0.8*\hmax*sin(-30)})node[pos=0.7,above]{$2$} ;

\draw[black!50!green, dashed, line width=0.8mm, ->] (0,0) --({0.8*\hmax*sin(\thet)},{0.8*\hmax*cos(\thet)}) ; 

\draw[black!50!green, dashed, line width=0.8mm, ->] ({0.7*\hmax*cos(-30)},{0.7*\hmax*sin(-30)}) to [bend right = 30] ({0.7*\hmax*cos(90-\thet)},{0.7*\hmax*sin(90-\thet)});

\draw [blue, thick, opacity = 1, ->, align = center](0.3*\hmax, 1*\hmax)node[right]{Pointing \\ constraint\\ cone} to [bend right = 30]  (0.1*\hmax, 0.85*\hmax)   ;

\end{tikzpicture}     
    \captionof{figure}{The thrust pointing constraint. Red: nonsaturating orientation. Green: saturating orientation. }
        \end{minipage}
        \begin{minipage}[b]{0.5\linewidth}
        \centering
        
    \begin{tikzpicture}[scale = 0.85]

\def \hmax{4cm}
\def \gam{45}
\def \H{3cm}

\draw (-\H,{\H/2*cos(60)}) -- ({\H + \H/2*sin(60)},{\H/2*cos(60)}) -- (\H ,{-\H/2*cos(60)}) -- ({-\H - \H/2*sin(60)},-{\H/2*cos(60)}) -- (-\H ,{\H/2*cos(60)}) -- cycle;

\draw[->](-\H,{\H/2*cos(60)}) -- (-\H ,{\H/2*cos(60)+ 0.7\H}) node[left]{$e_z$};
\draw[->](-\H,{\H/2*cos(60)}) -- (-\H + 0.7\H,{\H/2*cos(60)}) node[above]{$e_y$};
\draw[->](-\H,{\H/2*cos(60)}) -- ({-\H - 0.7\H/2*sin(60)} ,{\H/2*cos(60)- 0.7\H*cos(60)}) node[right]{$e_x$};

\fill [opacity = 1, white]({\hmax*cos(\gam)}, {\hmax*sin(\gam)}) -- (0,0) -- (-{\hmax*cos(\gam)}, {\hmax*sin(\gam)}) arc (180:360:{\hmax*cos(\gam)} and 0.1*\hmax);
\draw ({\hmax*cos(\gam)}, {\hmax*sin(\gam)}) -- (0,0) -- (-{\hmax*cos(\gam)}, {\hmax*sin(\gam)}) ;
\fill[opacity = 0.1] (0, {\hmax*sin(\gam)}) circle({\hmax*cos(\gam)} and 0.1*\hmax);
\draw (0, {\hmax*sin(\gam)}) circle({\hmax*cos(\gam)} and 0.1*\hmax);

\draw[thick] plot[mark=*, mark size=0.03\hmax] (0, 0) node[below]{$(x_f, y_f, 0)$};
\draw[thick] plot[mark=*, mark size=0.03\hmax] (\hmax/5, {\hmax*sin(\gam)}) coordinate (r0)  node[left]{$r_0$};
\draw[very thick]  plot[smooth] coordinates{(r0)  ({0.46*\hmax*cos(\gam)}, {0.5*\hmax*sin(\gam)}) ({0.1*\hmax*cos(\gam)}, {0.3*\hmax*sin(\gam)}) (0,0)};
\draw[->] (r0) -- (0.27*\hmax, {0.8*\hmax*sin(\gam)})node[right]{$v_0$};


\draw[gray] (0,0) --(\hmax/4,0); 
\draw[->] (\hmax/7,0) to [bend right = 30] ({0.15*\hmax*cos(\gam)}, {0.15*\hmax*sin(\gam)});
\draw (\hmax/6, 0.2\hmax)node[right]{$\gamma$} ;


\draw[->, thick] (\hmax/2,\hmax/3)node[right]{Optimal trajectory} to [bend right = 30] ({0.5*\hmax*cos(\gam)}, {0.6*\hmax*sin(\gam)});
\draw (\hmax/2,\hmax/3)  ;


\draw[align=center, ->] (0.7*\hmax,0.8*\hmax)node[right]{Glide-slope \\ constraint cone} to [bend right = 30] (\hmax/2,0.7*\hmax);

\draw[->, thick] ({0.45*\hmax*cos(\gam)}, {0.47*\hmax*sin(\gam)}) -- ({0.3*\hmax*cos(\gam)}, {0.62*\hmax*sin(\gam)})node[left]{$n$};

\end{tikzpicture}  
    \captionof{figure}{Glide-slope constraint}
    \vspace{30pt}
        \end{minipage}
    
    \end{figure}
    \end{center}
    \item a glide-slope constraint, \[ h(r) = z - \tan(\gamma) \Vert (x, y) \Vert \geq 0, \ \text {with} \ \gamma \in [0, \frac{\pi}{2}). \]
\end{itemize}

The landing problem will minimize a final cost : \[ \min J := \ell(t_f, X(t_f)),\]
where $t_f$ is the final time. 

Thus, the complete optimal control problem can be written
\begin{problem}\label{pb:completePb}
\[
\min   \ell(t_f, X(t_f)) \text{ such that}
\begin{cases} 
											\begin{array}{l}
													\left(r(\cdot), v(\cdot), m(\cdot)\right)      \in \R^3\times\R^3\times \R \;      \text{follows \eqref{eq:dyn}}, \\
													(r,v,m)(0)=\left(r_0, v_0 ,m_0\right), \\
													(z,v_z)(t_f)=\left(0, 0 \right), \\
													m(t) > m_e \; \forall t \in [0,t_f),\\
													u_{min} \leq ||u|| \leq u_{max},  \\
													\langle e_z, u \rangle \geq \Vert u \Vert      \cos(\theta),\\
													h(r) \geq 0. \\
											\end{array}
										
\end{cases}
\]

\end{problem}
\begin{rem}
Note that one can not have $m(t) = m_e$ for some $t\in [0, t_f)$. Indeed it would imply that $u=0$ on $[t,t_f]$, which makes it impossible to reach $z(t_f)=0$ with $v_z(t_f)=0$. Thus, the mass constraint $m(t) > m_e$ plays no role in the problem. 
\end{rem}

\begin{rem}
Note that the existence of optimal solutions for Problem 1 results from an adaptation of the results in \cite{Acikmese2007}. Indeed, the existence if proved in \cite{Acikmese2007} for the same problem formulation but without considering the pointing constraint. The proof consists in first convexifying the constraints and showing the existence of an optimal solution for the convex formulation of the problem \cite[Theorem $1$]{Acikmese2007}. This part of the reasoning still works when we add the convexified pointing constraint. Then, it is necessary to show that optimal solutions of the convexified problem are also optimal for the non convex formulation of the problem. This is done in \cite[Lemma $2$]{Acikmese2007} without the pointing constraint, and in \cite[Lemma $1$]{acikmese2013} when considering the pointing constraint.
\end{rem}

\subsection{Optimality conditions}
\label{se:optimality_conditions}

The necessary optimality conditions provided by the application of the Pontryagin Maximum Principle to the optimal control Problem \ref{pb:completePb} are given in this section. Let us firstly specify some notations.

\begin{itemize}
  \item For $p = (p_x, p_y, p_z) \in \R^3$, we set $\overline{p}=(p_x, p_y)$. Thus the glide-slope constraint writes as $h(r) = r_z - \tan(\gamma) \Vert \overline{r} \Vert$ and
      \begin{equation}\label{}
        n= \nabla h (r) = \begin{pmatrix} - \tan(\gamma)\frac{\overline r}{\Vert \overline r \Vert } \\ 1 \end{pmatrix}.
      \end{equation}
  
  \item The control set is $\U = \lbrace u \in \R^3 \ : \ u_{min} \leq \| u\| \leq u_{max} \ \text{and} \ \langle e_z, u \rangle \geq \| u\| \cos(\theta) \rbrace$.
\end{itemize}
We apply the version of the maximum principle given in \cite[Th.\ 9.5.1]{Vinter2010}, which in our case writes as follows.
Define the Hamiltonian of Problem \ref{pb:completePb} as
\begin{equation}\label{eq:H}
H (X,P,u,p^0) =  \langle p_r, v  \rangle + \langle p_v, \frac{T}{m} u - g  \rangle- p_m q \| u\|,
\end{equation}
where the adjoint vector $P = (p_r, p_v, p_m)$ belongs to $\R^3 \times \R^3 \times \R$ and $p^0 \in \R$.

Let $(X(\cdot),u(\cdot))$ be an optimal solution of Problem \ref{pb:completePb}. Then there exists a constant $p^0=0$ or $-1$, an absolutely continuous function $P(\cdot)$, and a nonnegative Borel measure $\mu$ on $[0,t_f]$, such that, writing
\begin{equation}
\label{eq:def_q}
Q(t)=\left( q_r(t) ,p_v(t), p_m(t) \right),  \qquad q_r(t)=p_r(t) - \int_{[0,t)} n(s)  \mu (ds),
\end{equation}
we have:
\begin{enumerate}
  \item $(P,p^0,\mu) \neq (0,0,0)$;
  \item $\mathrm{supp} \{\mu\} \subset \{t \in [0,t_f] \ : \  h(r(t))=0 \}$;
\item (dynamics of the adjoint vector) for a.e.\ $t \in [0,t_f]$,
\begin{equation}\label{}
\begin{cases}
\dot{p}_r(t) &= 0 ,\\
\dot{p}_{v}(t) &= - q_r (t),\\
\dot{p}_m(t) &=\frac{T}{m(t)^2} \left \langle p_v(t), u(t) \right \rangle;\\
\end{cases}
\end{equation}
  \item (maximization condition) for a.e.\ $t \in [0,t_f]$,
  \begin{equation}\label{}
    H(X(t), Q(t), p^0, u(t)) = \max \limits_{w \in \U} H(X(t), Q(t), p^0, w);
  \end{equation}
  \item (transversality condition) 
  \begin{equation}\label{eq:trCond}
    \max \limits_{w \in \U} H(X(t_f), Q(t_f), p^0, w) = -p^0\frac{\partial \ell}{\partial t}(t_f, X(t_f)).
  \end{equation}
\end{enumerate}

\section{General results}
\subsection{Solution of the optimal control problem}

Thanks to the application of the Pontryagin Maximum Principle, several results on the evolution of the control can be given. Those results will be necessary to deduce the general form of the control with respect to time. For $d = 1,2$, we denote by $\mathcal{S}^d$ the unit sphere of $\R^{d+1}$. Remind also that, for $p_v = (p_{v_x}, p_{v_y}, p_{v_z})$, we set $\overline{p}_v=(p_{v_x}, p_{v_y})$.

\begin{lem} \label{lem: sol_u}
Let $u(t)$, $t\in [0, t_f]$, be an optimal control of Problem \ref{pb:completePb}, and $P = (p_r, p_v, p_m) \in \R^3 \times \R^3 \times \R$ its adjoint vector.
Then, for any $t \in [0,t_f]$ such that $u(t) \neq 0$, there holds
\[
 \frac{u(t)}{\|u(t)\|} = d(t),
\]
where $d: [0, t_f] \rightarrow \mathcal{S}^2$ is a measurable function satisfying
\begin{equation}
d(t)  =
\begin{cases}
\frac{p_v(t)}{\Vert p_v(t) \Vert} \; & \normalfont\text{if}\, p_{v_z}(t) \geq \Vert p_v(t) \Vert \cos(\theta) \ \text{and} \ p_v(t) \neq 0,\\[1mm]
 \left( \sin(\theta) \frac{\overline p_v(t)}{\Vert \overline p_v(t) \Vert},   \cos(\theta) \right) \,  & \normalfont\text{if}\, p_{v_z}(t) < \Vert p_v(t) \Vert \cos(\theta) \ \text{and} \ \overline p_v(t) \neq 0,\\[1mm]
\Big( \sin(\theta) \delta,  \cos(\theta) \Big) \ \normalfont \text{with}\ \delta \in \mathcal{S}^1 \ & \normalfont\text{if}\ p_{v_z}(t) < \Vert p_v(t) \Vert \cos(\theta) \ \text{and} \ \overline p_v(t) = 0.
\end{cases}
\end{equation}
Moreover, set
\begin{equation}\label{eq:psi}
\Psi(t) = \frac{T}{m}\langle  p_v(t),  d(t)  \rangle - p_m(t) q.
\end{equation}
Then,
\begin{center} \normalfont
$
\Vert u(t) \Vert =
\begin{cases}
u_{max}\, &\text{if}\; \Psi(t) > 0,\\
u_{min} \,&\text{if}\; \Psi(t) < 0.
\end{cases}
$
\end{center}
\end{lem}

\begin{proof}[Proof of Lemma~\ref{lem: sol_u}]
If $u(t)$ is an optimal control on $[0,t_f]$, then the maximization condition of the Pontryagin Principle implies that, for almost all $ t \in [0,t_f]$, $u(t)$ maximizes
\[ \varphi (t,w) =    \langle p_v, \frac{T}{m} w  \rangle - p_m q \Vert w \Vert\]
among the $w \in \U$.
Making the change of variable $w = \alpha d$, with $\alpha =\Vert w \Vert$ and $d \in \mathcal{S} ^2$, to find $u(t)\in \U$ maximizing $\varphi$ amounts to find $\alpha$ and $d$ maximizing
\[ \varphi (t,w) = \alpha \left( \frac{T}{m} \langle  p_v,  d  \rangle - p_m q \right)\]
under the conditions $ u_{min} \leq \alpha \leq u_{max} $ and $ d \in \mathcal{D} = \lbrace d \in \mathcal{S}^2: \langle e_z , d \rangle \geq \cos(\theta) \rbrace.$
Let us write 
\[ 
\varphi(t,w) = \alpha \psi(t,d), \quad \hbox{where} \quad \psi(t,d) =   \frac{T}{m}\langle  p_v,  d  \rangle - p_m q .
\]
First,  $\alpha \geq 0 $ implies that\[ \max \varphi = \max \limits_\alpha \left[ \alpha ( \max \limits_d \psi ) \right].\]
Then, the maximum of $\varphi$ is attained for $\alpha$ satisfying
\[
\begin{cases}
\alpha=u_{max}\, &\text{if}\, \max(\psi) > 0,\\
\alpha=u_{min} \,&\text{if}\, \max(\psi) < 0,
\end{cases}
\]
and for $d$ solution of the problem
\[\max \limits_d \psi =  -p_mq + \frac{T}{m} \max \limits_{d \in \mathcal{D}} \langle p_v, d \rangle. \]
It is a matter of an exercise to check that the solution $d$ of this maximization problem satisfies the statement of the lemma (for the sake of completeness we give the proof in Appendix, in Lemma~\ref{le:kkt}). Lemma~\ref{lem: sol_u} follows by setting $\Psi=\max \limits_d \psi$. 
\end{proof}

Thus, the norm of the control is mainly determined by the sign of the switching function $\Psi$ and takes only the values $u_{min}$ and $u_{max}$ when $\Psi$ is non zero (\emph{bang arcs}). Therefore, we need to analyze the variations of $\Psi$, which we will do now. Note however that, on time intervals where $\Psi$ is zero, we have \emph{singular arcs}. In that case the norm of $u$ is not given anymore by Lemma~\ref{lem: sol_u}, a further study is necessary, see Section~\ref{se:singular_traj}.

\begin{rem}
\label{re:dcontinuous}
The direction $d$ of the control is more regular than just measurable. Indeed let $I \subset [0, t_f]$ be the closed set such that, for every $t \in I$,
\begin{equation}\label{eq:non_hyp}
p_v(t) = 0 \qquad \hbox{or} \qquad  p_{v_z} (t) < 0 \ \hbox{ and } \ \overline p_v(t) = 0.
\end{equation}
Equivalently, for every
$t \in [0,t_f]\backslash I$, there holds
\begin{equation}\label{eq:hyp}
p_v(t) \neq 0 \qquad \hbox{and} \qquad \hbox{if } p_{v_z} (t) < \| \overline p_v(t) \| \cos \theta, \ \hbox{then } \ \overline p_v(t) \neq 0,
\end{equation}
which implies that for these times $d(t)$ is defined as a function of $p_v(t)$. This function being locally Lipschitz, the function $d(\cdot)$ is absolutely continuous on every subinterval of $[0,t_f]\backslash I$. In particular $d(\cdot)$ is continuous on $[0,t_f]\backslash I$ while it may be discontinuous on $I$.
\end{rem}

\subsection{Variations of $\Psi$}

\begin{prop}
\label{le:prop}
Consider an optimal trajectory on $[0,t_f]$. Then:
\begin{enumerate}
\item $\Psi$ is absolutely continuous on $[0, t_f]$;
  \item at a.e.\ $t \in [0,t_f]$, the time derivative of $\Psi$ is
  \begin{equation}\label{eq:dPsi}
    \dot{\Psi}(t) = - \frac{T}{m(t)} \langle q_r(t), d (t) \rangle;
  \end{equation}
  \item $t \mapsto \langle q_r(t), d(t) \rangle$ is a nonincreasing function on a full measurement subset of $[0,t_f)$.
\end{enumerate}
\end{prop}
\begin{proof}[Proof of Points 1.\ and {2.}]
Consider an optimal trajectory on $[0,t_f]$ with adjoint vector $P = (p_r, p_v, p_m)$.
Recall that the switching function $\Psi$ is given by
\begin{equation}
\Psi(t) = \frac{T}{m}\langle  p_v(t),  d(t)  \rangle - p_m(t) q,
\end{equation}
the fonction $d(t)$ being defined in Lemma~\ref{lem: sol_u}. Note first that, for every  $t \in [0, t_f]$, there holds
\[ \sca{p_v}{d}(t) =
\begin{cases}
 \| p_v(t) \|  &\text{if}\, p_{v_z}(t) \geq \cos(\theta) \| p_v(t) \|,\\
 \| \overline p_v(t) \|\sin(\theta) + p_{v_z} \cos(\theta) &\text{otherwise}.
\end{cases}
\]
Thus $\sca{p_v}{d}$ is a Lipschitz function of $p_v$. The latter being an absolutely continuous function on $[0,t_f]$, $\sca{p_v}{d}$ is absolutely continuous too, which proves Point \emph{1.}

Moreover,  a simple computation shows
\[
\dot{\Psi} = \frac{T}{m} \frac{ d\langle p_v, d \rangle}{dt} \quad \hbox{a.e.\ on } [0,t_f],
\]
so we are left with the computation of the time derivative of $\sca{p_v}{d}$.

Let us introduce the subsets $I_1 = \lbrace t \in [0,t_f] \ : \ p_v(t) = 0 \rbrace$,  $I_2 = \lbrace t \in [0,t_f] \ : \ \overline p_v(t) = 0 \ \text{and} \ p_{v_z}(t) <0 \rbrace$, and $I = I_1 \cup I_2$. As noticed in Remark~\ref{re:dcontinuous}, the function $d(\cdot)$ is absolutely continuous, and so is differentiable a.e., on $[0, t_f] \backslash I$. Hence on this subset we can compute directly the time derivative of $\sca{p_v}{d}$ as follows:
\begin{itemize}
  \item when the pointing constraint is not active, there holds $\langle p_v, d \rangle = \Vert p_v \Vert$ and then
\[\frac{
d\langle p_v, d \rangle}{dt} = \frac{ d \Vert p_v \Vert}{dt} = - \frac{  \langle q_r, p_v \rangle}{ \Vert p_v\Vert} = - \langle q_r, d \rangle.
\]

  \item when the pointing constraint is active, $\langle p_v,d \rangle = \sin(\theta) \Vert \overline{p}_v \Vert  + \cos(\theta)p_{v_z}$
and we obtain
\begin{equation}\label{}
   \frac{ d \langle p_v,d \rangle}{dt} = \sin(\theta) \left(\frac{- \langle \overline q_r ,\overline p_v \rangle}{\Vert \overline p_v \Vert} \right) -\cos(\theta)q_{z} = -\langle q_r,d \rangle.
\end{equation}
\end{itemize}
As a consequence, \eqref{eq:dPsi} holds a.e.\ on $[0, t_f] \backslash I$.

Now, since $p_v=0$ and $\langle p_v, d \rangle=0$ on $I_1$, then $\dot p_v = - q_r = 0$ and $\frac{d\langle p_v, d \rangle}{dt}=0$ a.e.\ on $I_1$ (see for instance~\cite[Lemma~3.10]{chi06}), which implies that \eqref{eq:dPsi} holds a.e.\ on $I_1$. The same argument on $\overline p_v $ instead of $p_v$ shows that \eqref{eq:dPsi} holds a.e.\ on $I_2$ too, and so on $I$. This concludes the proof of Point \emph{2}.
\end{proof}

\begin{lem}
\label{le:hypo2}
Consider a subinterval $(t_1, t_2)$ of $[0, t_f] \backslash I$.
Then, $t \mapsto \langle q_r(t), d(t) \rangle$ is a nonincreasing function on $(t_1, t_2)$.
\end{lem}

In order to prove Lemma \ref{le:hypo2}, we will show that $\langle q_r(t),d (t) \rangle$ is locally nonincreasing, i.e., for any $t_0 \in [0,t_f)$ and any $t>t_0$ close enough, there holds $\langle q_r(t),d (t) \rangle \leq \langle q_r(t_0),d (t_0) \rangle$. We first prove two preliminary lemmas.

\begin{lem}
\label{le:signPrd}
For a.e.\ $t\in (t_1, t_2) \subset [0, t_f] \backslash I$, there holds
\begin{equation}\label{}
  \langle \dot p_v(t),\dot d (t) \rangle \geq 0.
\end{equation}
\end{lem}

\begin{proof}
Recall first (see Remark~\ref{re:dcontinuous}), that  $d(\cdot)$ is absolutely continuous on $(t_1, t_2)$ and that on such an interval,
\begin{equation}
  d(t) = \frac{ p(t)}{\Vert  p(t) \Vert} \quad \hbox{or} \quad d(t) = \left( \sin(\theta) \frac{\overline p_v(t)}{\Vert \overline p_v(t) \Vert},   \cos(\theta) \right).
\end{equation}
Thus, for a.e.\ $t\in (t_1, t_2)$, we have either
\begin{equation}\label{}
 \langle \dot p_v(t),\dot d (t) \rangle
= \frac{1}{\Vert  p_v(t)\Vert} \left(  \Vert \dot p_v(t) \Vert^2 - \frac{\langle \dot p_v(t),  p_v(t) \rangle^2}{\Vert  p_v(t) \Vert^2}\right),
\end{equation}
or,
\begin{equation}\label{}
\langle \dot p_v(t),\dot d (t) \rangle
= \frac{\sin(\theta)}{\Vert \overline{p}_v (t) \Vert} \left( \Vert  \dot{\overline{p}}_v(t)  \Vert^2 - \frac{ \langle \dot{\overline{p}}_v(t), \overline{p}_v (t) \rangle^2}{\Vert \overline{p}_v(t) \Vert^2} \right).
\end{equation}
In both case it is nonnegative due to Cauchy-Schwarz inequality.
\end{proof}

Note that moreover $\langle \dot p_v(t),\dot d (t) \rangle=0$ if and only if $p_v(t)$ and $\dot{p}_v(t)$ are collinear when $d(t) = \frac{ p(t)}{\Vert  p(t) \Vert}$, and if and only if $\overline p_v(t)$ and $\dot{\overline{p}}_v(t)$ are collinear otherwise.

\begin{coro}
\label{le:coro_free_arcs}
If the state constraint is not active (i.e.\ $h(r(t))\neq 0$) on $(t_1,t_2) \subset [0, t_f] \backslash I$, then $\langle q_r(t),d (t) \rangle$ is differentiable a.e.\ on $(t_1,t_2)$ and
\begin{equation}\label{}
\frac{d}{dt}\langle q_r(t),d (t) \rangle \leq 0.
\end{equation}
Moreover this derivative is zero if and only if $p_v(t)$ and $q_r$ are collinear when $d(t) = \frac{ p_v(t)}{\Vert  p_v(t) \Vert}$, and if and only if $\overline p_v(t)$ and $\overline q_r$ are collinear otherwise.
\end{coro}

\begin{proof}
On $(t_1,t_2)$, the measure $\mu$ is zero, so that $q_r$ is constant and $\dot p_v(t)=- q_r$. Therefore, the sign of the derivative
\begin{equation}\label{}
\frac{d}{dt}\langle q_r(t),d (t) \rangle =  \langle q_r, \dot d(t) \rangle = - \langle \dot p_v(t), \dot d(t) \rangle
\end{equation}
is given by Lemma~\ref{le:signPrd}.
\end{proof}

\begin{lem}
\label{le:signe_nd}
Let $(t_1,t_2) \subset [0, t_f] \backslash I$ and let $t_0 \in (t_1, t_2)$ such that $h(r(t_0))=0$. Then
\begin{equation}\label{}
  \langle n(t_0),d(t_0) \rangle >0, \qquad \normalfont\text{where} \quad n(t)=\begin{pmatrix} - \tan(\gamma)\frac{\overline r (t)}{\Vert \overline r (t)\Vert } \\ 1 \end{pmatrix}.
\end{equation}
\end{lem}

\begin{proof}
Set $h[t]=h(r(t))$. Since $h(r)=r_z - \tan(\gamma) \Vert \overline{r} \Vert$, the function $t \mapsto h[t]$ is $C^1$ at times such that $\overline r(t) \neq 0$. Its derivative
\begin{equation}\label{}
\dot h[t]= \langle n(t),v(t) \rangle,
\end{equation}
is absolutely continuous and, for a.e.\ $t$,
\begin{equation}\label{eq:ddoth}
\ddot h[t]= \frac{T}{m(t)} \|u(t)\| \sca{n(t)}{d(t)}- g_0 + \langle \dot n(t),v(t) \rangle.
\end{equation}
Moreover, due to Cauchy-Schwarz inequality, we have
\begin{equation}\label{eq:nv}
 \langle \dot n(t),v(t) \rangle = \frac{\tan(\gamma)}{\Vert \overline{r} (t) \Vert} \left( - \Vert \overline{v}(t) \Vert^2 + \frac{ \langle \overline{r}(t), \overline{v} (t) \rangle^2}{\Vert \overline{r}(t) \Vert^2}\right) \leq 0.
\end{equation}
Note that $\overline r(t_0) \neq 0$. Indeed, otherwise $h(r(t_0))=0$ would imply $r(t_0) = 0$ and $v_z(t_0) = 0$, as $z$ has reached a minimum. Then, as $z(t) = o(t-t_0)$ and $ h(t) = - \tan(\gamma) | t- t_0 | \|\overline v(t_0) \| + o(t-t_0)$,  $h(t) \geq 0$ would imply that $\overline v(t_0) = 0$ and thus $v(t_0) =0$, a contradiction with $t_0 \neq t_f$. Thus $\dot h[t]$ is differentiable at $t_0$.

Since $h[t]\geq 0$ on $[0,t_f]$, $t_0$ is a minimum of $h[t]$, therefore $\dot h[t_0]=0$. Moreover, there exists a sequence of times $t^k>t_0$ converging to $t_0$ such that $\dot h[t^k]\geq 0$. Thus
\begin{equation}\label{}
  \int_{t_0}^{t^k} \ddot h[t] dt = \dot h[t^k] \geq 0.
\end{equation}
Using \eqref{eq:ddoth} and \eqref{eq:nv}, this implies
\begin{equation}\label{}
 \int_{t_0}^{t^k} \frac{T}{m(t)} \|u(t)\| \sca{n(t)}{d(t)} dt \geq g_0 (t^k-t_0).
\end{equation}
The function $d(t)$ being continuous (see Remark~\ref{re:dcontinuous}), $\sca{n(t)}{d(t)}$ is continuous too, and the above inequality implies that $\sca{n(t_0)}{d(t_0)}$ is  nonnegative and then
\begin{equation}\label{}
g_0  \leq \frac{1}{t^k-t_0}\int_{t_0}^{t^k} \frac{T}{m(t)} \|u(t)\| \sca{n(t)}{d(t)} dt \leq \frac{T}{m_e} u_{max}\max_{[t_0, t^k]}\sca{n(t)}{d(t)},
\end{equation}
and $ \max_{[t_0, t^k]}\sca{n(t)}{d(t_0)}$ tends toward $ \sca{n(t_0)}{d(t_0)}$ when $t^k$ converges to $t_0$,
which gives the result.
\end{proof}
\begin{rem}
The above inequality shows also that if $\|u(t)\|$ has a limit $\alpha$ when $t$ tends toward $t_0$, then the limit $\alpha$ satisfies $ T \alpha \geq m_e g_0 \cos(\gamma)$.
\end{rem}
\begin{proof}[Proof of Lemma~\ref{le:hypo2}]
Let $t_0 \in [0,t_f)$. If $h[t_0]\neq 0$, then $h[t]\neq 0$ in a neighborhood of $t_0$. It then results from Corollary~\ref{le:coro_free_arcs} that $t \mapsto \langle q_r(t),d (t) \rangle$ is nonincreasing near $t_0$.

We are left with the case $h[t_0] = 0$. Recall that, for $t>t_0$, we have from the maximum principle
\begin{equation}
q_r(t)=q_r(t_0) - \int_{[t_0,t)} n(s)  \mu (ds),
\end{equation}
and that the function $d(\cdot)$ is absolutely continuous on $(t_1,t_2)$, so that, for $t$ near $t_0$, $d(t)-d(t_0)=\int_{t_0}^t \dot d(s) ds$. 
Thus, for $t>t_0$ close enough to $t_0$ we have 
\begin{equation}\label{eq:I123}
\langle q_r(t),d (t) \rangle - \langle q_r(t_0),d (t_0) \rangle  =  I_1 + I_2 - I_3,
\end{equation}
where
\begin{equation} 
\begin{array}{c}
\displaystyle I_1 = \int_{t_0}^{t}\langle q_r(s), \dot d(s) \rangle ds, \quad I_2 = \int_{t_0}^{t} \left\langle \int_{[t_0,s)} n(\sigma)  \mu (d\sigma) , \dot d(s) \right\rangle ds, \\[3mm]
\displaystyle I_3 =  \int_{[t_0,t)} \langle n(s),d (t) \rangle  \mu (ds).
\end{array}
\end{equation}
First, by Lemma~\ref{le:signPrd} we have $I_1 \leq 0$. Second, $p_v(t)$ and $q_r(t)$ are  bounded near $t_0$, and so $\dot d(t)$ is bounded too. Since $\mu$ is a positive measure, we obtain
\begin{equation}
I_2 \leq (t-t_0) \mathit{Cst}  \int_{[t_0,t)} \mu (ds).
\end{equation}
Third, by Lemma~\ref{le:signe_nd} and the continuity of $\sca{n(t)}{d(t)}$, we obtain
\begin{equation}
I_3  \geq \frac{1}{2} \langle n(t_0),d (t_0) \rangle  \int_{[t_0,t)} \mu (ds).
\end{equation}
As a consequence, $I_2-I_3 \leq 0$ for $t-t_0$ small enough. Thus $\langle q_r(t),d (t) \rangle \leq \langle q_r(t_0),d (t_0) \rangle$ for $t>t_0$ close enough to $t_0$, which concludes the proof.
\end{proof}

\begin{proof}[Proof of Proposition~\ref{le:prop}-Point {3.}]
Let us show that for every $t_0 \in [0, t_f]$, $t \mapsto \sca{q_r}{d}$ is nonincreasing almost everywhere in the neighbourhood of $t_0$. If \eqref{eq:hyp} is true at $t_0$, the result is given by Lemma~\ref{le:hypo2}. Otherwise, let us consider the case when $\overline p_v(t_0) =0$ and $p_{v_z}(t_0) < 0$. The case when $p_v(t_0) = 0 $ can be treated in the same way.
Firstly, it should be noted that
\begin{enumerate}
    \item from $q_z(t) = p_z(t_0) - \int _{[0,t)} \mu(ds)$, $q_z$ is a nonincreasing function on $[0, t_f]$;
    \item from $\overline q_r(t) = \overline p_r(t_0) + \int _{[0,t)} \tan(\gamma) \frac{\overline r}{\| \overline r \|} \mu(ds)$, as $\overline r(t_0) \neq 0$ if $h(t_0) = 0$ and $q_r$ is constant if $h(t_0) \neq 0$, $q_x$ and $q_y$ are monotonous in a neighbourhood of $t_0$;
    \item $\overline p_v(t) = - \int _{t_0}^t \overline q_r(s) ds.$
\end{enumerate}
Recall that we denote by $\overline d$ the vector composed of the two first coordinates of $d$.
Then, in a neighbourhood $(t_-, t_+)$ of $t_0$, with $t_- \leq t_+$,\[ \sca{q_r}{d} = \sca{\overline q_r}{\overline d} + \cos(\theta)q_z.\] As $q_z$ is nonincreasing, let us show that $\sca{\overline q_r}{\overline d}$ is nonincreasing.
From 2., on each interval $J = (t_-, t_0)$ or  $(t_0, t_+)$,
\begin{itemize}
    \item either $\overline q_r = 0$ and $\sca{\overline q_r}{\overline d} = 0$ is constant;
    \item or $ \overline q_r(t) \neq 0 $ $\forall t \in J $, thus $\overline p_v(t) \neq 0 $ and $ \overline d = \sin(\theta) \frac{\overline p_v}{\| \overline p_v \|}$. In this case, \[ \sca{\overline q_r}{\overline d} = - \frac{\sin(\theta)}{\| \overline p_v(t) \|} \int _{t_0}^t \sca{\overline q_r(s) }{\overline q_r(t)} ds.
    \]
We can assume, by reducing the interval $J$ if necessary, that for every $t$ and $s$ in $J$, $\text{sign}(q_x(t)) = \text{sign}(q_x(s))$ and $\text{sign}(q_y(t)) = \text{sign}(q_y(s))$. It implies that $\sca{\overline q_r(s)}{\overline q_r(t)} \geq 0$, and then $\sca{\overline q_r}{\overline d}$ is nonincreasing on $J$ and is of the same sign as $(t_0 - t)$.
\end{itemize}
Combining these results, we obtain that $\sca{\overline q_r}{\overline d}$ is nonincreasing on $(t_-,t_0) \cup (t_0, t_+)$, which concludes the proof of the proposition.
\end{proof}

\begin{proof}[Proof of Theorem \ref{th:BOB}]
Consider an optimal trajectory on $[0,t_f]$. From Lemma \ref{lem: sol_u}, the norm of the control has the following expression depending on the sign of $\Psi$:
\begin{center}
$
\Vert u(t) \Vert =
\begin{cases}
u_{max}\, &\text{if}\; \Psi(t) > 0,\\
u_{min} \,&\text{if}\; \Psi(t) < 0 \\
\end{cases}
$
\end{center}
and it is singular if $\Psi(t) = 0$.

Let us show that $\Psi$ can change of sign at most two times or be constantly zero on at most one interval. From Proposition~\ref{le:prop}-Point \textit{2}, the sign of the derivative of $\Psi$ is the opposite of the one of $\sca{q_r}{d}$. Moreover, from Proposition~\ref{le:prop}-Point \textit{3}, $\sca{q_r}{d}$ is non increasing, thus it may be zero on at most one interval denoted $[\overline t_1, \overline t_2] \subset [0, t_f]$. Then, on $[0, \overline t_1]$ $\sca{q_r}{d} >0$ and $ \dot \Psi <0$, and on $[\overline t_2, t_f]$ $\sca{q_r}{d} <0$ and $ \dot \Psi >0$.
\begin{itemize}
\item Assume that $\Psi(t) \neq 0$ for $t \in [\overline t_1, \overline t_2]$. Then, on $[0, \overline t_1]$( respectively $[\overline t_2, t_f]$)  as $\Psi$ is continuous (from Proposition~\ref{le:prop}-Point \textit{1})  and decreasing (resp. increasing), it may be zero at most one time. Finally, on $[0, t_f]$, it may be zero and change of sign at most two times. If $\Psi(t) >0$ on  $[\overline t_1, \overline t_2]$, then it is strictly positive everywhere on $[0, t_f]$. If $\Psi(t) <0$ on  $[\overline t_1, \overline t_2]$, let us define $t_1 \in [0, \overline t_1)$ (resp. $t_2 \in (\overline t_2, t_f]$) such that $\Psi(t_1) = 0$ or $t_1 = 0$ (resp. $t_2 = t_f$). Then $\Psi(t) >0$ on $[0,t_1) \cup (t_2, t_f]$ and $\Psi(t) <0$ on $(t_1, t_2)$.

\item Assume now that $\Psi$ crosses zero on  $[\overline t_1, \overline t_2]$. As $\dot \Psi(t)$ is zero on that interval, then $\Psi(t) = 0$ for all $t \in [\overline t_1, \overline t_2]$. Moreover, as it is decreasing on $[0, \overline t_1]$ and increasing on $[\overline t_2, t_f]$, $\Psi(t)>0$ for all $t \in [0, \overline t_1) \cup (\overline t_2, t_f]$.
\end{itemize}
From this study on the sign of $\Psi$, we deduce the form of the control as stated in the theorem. 
\end{proof}

\subsection{Results on singular trajectories}
\label{se:singular_traj}
We have seen in Theorem~\ref{th:BOB} that optimal trajectories may contain singular arcs. The next lemmas show that such trajectories have a particular form, namely on singular arcs some adjoint vectors are collinear and the pointing constraint is active. We will show later (Lemma \ref{le:sing_ngen}) that in the case $\gamma = 0$, these properties happen only if the initial conditions are in very specific positions, so that singular arcs do not appear for generic initial conditions. We did not try to prove the result for $\gamma \neq 0$, but it is reasonable to believe that the results still holds (no singular arc for generic initial conditions).

\begin{lem}
\label{le:sing}
Consider an optimal trajectory on $[0,t_f]$ and an interval $(t_1',t_2') \subset [0,t_f] \backslash I$.
Then, if $\Psi(t)=0$ on $(t_1', t_2')$, then $\int_{[t_1',t_2')} \mu (ds)=0$. Moreover if $d(t) = \frac{ p_v(t)}{\Vert  p_v(t) \Vert}$, then $p_v(t)$ and $q_r(t)$ are collinear, otherwise  $\overline p_v(t)$ and $\overline q_r(t)$ are collinear.
\end{lem}
\begin{proof}
Assume $\Psi(t)=0$ on $(t_1',t_2')$. Then $\dot \Psi(t)=0$ a.e.\ on $(t_1',t_2')$, and it follows from Proposition~\ref{le:prop} that $\langle q_r(t),d (t) \rangle$ is constantly equal to zero on $(t_1',t_2')$. If the constraint is not active on $(t_1',t_2')$, the conclusion follows from Corollary~\ref{le:coro_free_arcs}. If the constraint is active at some time $t_0' \in (t_1',t_2')$, then Formula~\eqref{eq:I123} implies that both $I_1$ and $I_2-I_3$ must be zero near $t'_0$. The latter implies
\begin{equation}
\langle n(t'_0),d (t'_0) \rangle  \int_{[t'_0,t)} \mu (ds)=0
\end{equation}
which again gives the conclusion since $\langle n(t'_0),d (t'_0) \rangle>0$ by Lemma~\ref{le:signe_nd}.
\end{proof}

\begin{lem} \label{le:sing_d}
Assume that the final cost $\ell$ verifies $\frac{\partial \ell}{\partial t} \geq 0$. If $d(t) = \frac{p_v}{\|p_v\|} $ on an interval $(t_1', t_2') \subset [0, t_f] \backslash I$, then the trajectory is not singular on that interval.
\end{lem}
\begin{proof}
If the trajectory is singular, then by Lemma \ref{le:sing}, $q_r$ and $p_v$ are collinear. Moreover, when the pointing constraint is not active,
$\dot{\Psi} = 0$ implies from \eqref{eq:dPsi} that $\sca{q_r}{d} = \frac{\sca{q_r}{p_v}}{\|p_v\|} = 0$. Thus $q_r = 0$. 
Now, from \eqref{eq:H}, which is equivalent to\[ H(X,P,p^0,u) = \|u\|\Psi + \sca{q_r}{v} - p_{v_z}g_0, \] the expression of the Hamiltonian becomes \[ H(X,P,p^0,u) =  - p_{v_z}g_0.\] 
The transversality condition \eqref{eq:trCond} and our condition on $\ell$ then imply that $p_{v_z} \leq 0$, which contradicts the assumption that $d$ is in the cone.
\end{proof}
\begin{rem}
The assumption $\frac{\partial \ell}{\partial t} \geq 0$ means that the cost penalizes the total duration of the trajectory. It is a very natural assumption and is verified for the vast majority of the practical problems with free final time.
\end{rem}

\subsection{Structure of the optimal control when considering the effect of the atmosphere}
The result of Theorem \ref{th:BOB} can be extended when considering the effect of an atmosphere. For the sake of simplicity, we do not consider here the pointing constraint. We model the effect of an atmosphere at low altitude by a constant pressure applying a force in the direction of the launcher, which introduces a new term $\frac{\sigma}{\|u\|}$ in the equation of $\dot v$. The new dynamics are
\begin{equation}\label{eq:dynSP}
\begin{cases}
\dot{r} &= v, \\
\dot{v} &= \left(T -\frac{\sigma}{\|u\|}\right)\displaystyle{\frac{u}{m}}  - g , \\
\dot{m} &= -q \| u \|,\\
\end{cases}
\end{equation}
where $\sigma$ is a parameter which is assumed to be constant and depends on the engine nozzle exit area and on the atmospheric pressure.
We make the following assumption.
\begin{assum} \label{as:SP}  \normalfont
The net thrust is always positive: $Tu_{min} \geq \sigma$.
\end{assum}
The optimal control problem is formulated as follows. 
\begin{problem}\label{pb:completePbSP}
\[
\min   \ell(t_f, X(t_f)) \text{ such that}\]
\[
\begin{cases} 
	\begin{array}{l}
			\left(r(\cdot), v(\cdot), m(\cdot)\right)      \in \R^3\times\R^3\times \R \;      \text{follows \eqref{eq:dynSP}}, \\
			(r,v,m)(0)=\left(r_0, v_0 ,m_0\right), \\
			(z,v_z)(t_f)=\left(0, 0 \right), \\
			m(t) > m_e \; \forall t \in [0,t_f),\\
			u_{min} \leq \|u\| \leq u_{max}.  \\
	\end{array}
\end{cases}
\]

\end{problem}
The main result of this section studying Problem \ref{pb:completePbSP} is stated in the next theorem.
\begin{theorem}\label{th:BOBSP}
Consider an optimal trajectory on $[0,t_f]$. Then, the control $u(t)$ is in the Max-Min-Max or the Max-Singular-Max  form, i.e. there exists $t_1$ and $t_2$ with $0 \leq t_1 \leq t_2 \leq t_f$ such that 
\[\|u\|(t) =
\begin{cases}
& u_{max} \; \normalfont\text{if} \; t \in [ 0, t_1)\cup (t_2, t_f], \\
& u_{min} \;  \normalfont\text{or singular if} \; t \in [t_1, t_2].\\
\end{cases}
\]
\end{theorem}
The following developments will be dedicated to the proof of Theorem \ref{th:BOBSP}.
The optimality conditions are similar to those for Problem \ref{pb:completePb}, we only detail here the changes. 
We define the Hamiltonian of Problem \ref{pb:completePbSP} as 
\begin{equation}\label{eq:HSP}
H (X,P,u,p^0) =  \langle p_r, v  \rangle + \left\langle p_v, \left(T -\frac{\sigma}{\|u\|}\right)\displaystyle{\frac{u}{m}}   - g  \right\rangle- p_m q \Vert u \Vert.
\end{equation}
The dynamics of the adjoint vector are 
\begin{equation}\label{eq:dynpSP}
\begin{cases}
\dot{p}_r(t) &= 0 ,\\
\dot{p}_{v}(t) &= - q_r (t),\\
\dot{p}_m(t) &=\frac{1}{m(t)^2} \left \langle p_v(t), \left(T -\frac{\sigma}{\|u(t)\|}\right)\displaystyle{u(t)}  \right \rangle.\\
\end{cases}
\end{equation}

\begin{lem}\label{le:kktSP}
The optimal control of Problem \ref{pb:completePbSP} is expressed as follows
\[ 
\|u(t)\| = 
\begin{cases}
u_{max} \ \text{if} \ \Psi(t) >0,\\
u_{min} \ \text{if} \ \Psi(t) <0,\\
\end{cases}
\]
where 
\[ \Psi(t) = \frac{T}{m(t)}\|p_v(t)\| - p_m(t)q, \]
and for all $t$ such that $u(t) \neq 0$, 
\[ \frac{u(t)}{\|u(t)\|} = \frac{p_v(t)}{\|p_v(t)\|} .\]
\end{lem}
\begin{proof}
If $u(t)$ is an optimal control on $[0,t_f]$, then the maximization condition of the Pontryagin Principle implies that, for almost all $ t \in [0,t_f]$, $u(t)$ maximizes
\[ \varphi (t,w) =   \langle p_v, \left(T -\frac{\sigma}{\|w\|}\right)\displaystyle{\frac{w}{m}}   \rangle- p_m q \Vert w \Vert\]
among the $w \in \U$.
Making the change of variable $w = \alpha d$, with $\alpha =\Vert w \Vert$, to find $u(t)\in \U$ maximizing $\varphi$ amounts to find $\alpha$ and $d$ maximizing
\[ \varphi (t,\alpha, d) = \alpha \left( \frac{T}{m} \langle  p_v,  d  \rangle - p_m q \right) - \frac{\sigma}{m}\langle  p_v,  d  \rangle \]
under the conditions $ u_{min} \leq \alpha \leq u_{max} $ and $d \in \mathcal{S} ^2$. It is clear that this maximization yields to the statement of the lemma.

\end{proof}

\begin{lem}\label{le:propSP}
Consider an optimal trajectory on $[0,t_f]$. Then:
\begin{enumerate}
\item $\Psi$ is absolutely continuous on $[0, t_f]$;
  \item $t \mapsto \left \langle q_r(t), \frac{p_v(t)}{\|p_v(t)\|} \right \rangle$ is a nonincreasing function on a full measurement subset of $[0,t_f)$.
\end{enumerate}
\end{lem}
\begin{proof}
We do not detail the proof here as it is very similar to the one of Proposition \ref{le:prop}-Points $1$,$3$.
\end{proof}
\begin{proof}[Proof of Theorem \ref{th:BOBSP}]
Let study the sign of $\Psi(t)$. Its derivative is expressed by
\[
  \dot \Psi(t) =  \frac{1}{m(t)} \left( \frac{\sigma q}{m(t)} \|p_v\| - T \left \langle q_r(t), \frac{p_v(t)}{\|p_v(t)\|} \right \rangle \right), 
  \]
 and can be written as follows 
\[
  \dot \Psi(t) =   K(t)  \Psi(t) + \frac{1}{m(t)}f(t) , 
  \]
  where $K$ and $f$ are functions expressed by
  \[ K(t) = \frac{\sigma q}{Tm(t)} \qquad \text{and} \qquad f(t) = \frac{\sigma q^2}{T}p_m(t) - T\left \langle q_r(t), \frac{p_v(t)}{\|p_v(t)\|} \right \rangle .
  \]
  We deduce that 
  \[ \frac{d}{dt} \left[\Psi(t)e^{-\int_0^t K(s) ds} \right] = \frac{e^{-\int_0^t K(s) ds}}{m(t)}f(t).\]
 
 Using the expression of the optimal control given in Lemma \ref{le:kktSP} and from \eqref{eq:dynpSP}, we obtain that the derivative of $p_m$ is expressed by $\frac{\|p_v\|}{m(t)^2}\left(T\|u\| - \sigma \right)$; therefore, from Assumption \ref{as:SP}, $\dot p_m(t) >0$ and $p_m$ is increasing. From Lemma \ref{le:propSP} $\left \langle q_r(t), \frac{p_v(t)}{\|p_v(t)\|} \right \rangle $ is nonincreasing, thus we deduce that $f(t)$ is increasing and can change of sign at most one time. Consequently, $\Psi(t)e^{-\int_0^t K(s) ds}$ can changes of sign at most two times, and in that case it is positive, then negative and then positive, and $\Psi$ has the same property.
\end{proof}

\section{Specific results with pointing and altitude constraints}
We are now interested in the particular case of an altitude constraint, i.e.\ the case where $\gamma = 0 $. In that case, $n = e_z$, so $\overline q_r = \overline p_r$ is constant even when $h(r) = 0$.

\subsection{Singular arcs }
Despite Theorem \ref{th:BOB} does not exclude the existence of singular trajectories, it seems that they do not appear in general. They seldom arise when solving landing problems numerically, and the next lemma states that they do not exist in the case of an altitude constraint ($\gamma = 0$) if the initial conditions are generic, which is defined in Lemma \ref{le:sing_ngen}. The result of Lemma \ref{le:sing_ngen} has been proved under the following assumption.
\begin{assum}\label{as:pinpoint} \normalfont
We assume that the final position and velocity are fixed and null, i.e.\  $r(t_f) = r_f = 0$ and $v(t_f) = v_f = 0$. 
\end{assum}
\begin{lem}\label{le:sing_ngen}
If the optimal trajectory contains a singular arc, then the initial conditions are such that $( x, y)(0)$ and $( v_x, v_y )(0)$ are collinear. Consequently, for generic initial conditions there are no singular arcs in the optimal trajectories.
\end{lem}
\begin{proof}
Consider a trajectory containing a singular arc on an interval $[t_1', t_2']$. According to Lemma \ref{le:sing}, $\overline p_v(t_1')$ and $\overline p_r$ are collinear. 
Set $\overline d=(d_x,d_y)$, $\overline r = ( x, y )$ and $\overline v = ( v_x, v_y)$. 
Since $(r(t_f), v(t_f)) = (0,0)$, $\overline d$ has the same direction as $\overline p_v$ and as $\dot{\overline v}$ by the dynamics of the vehicle. We conclude that $\overline r(0)$ and $\overline v(0)$ are collinear.

\end{proof}

Lemma \ref{le:sing_ngen} helps to explain why singular trajectories seem to be rare and are often dismissed in the literature. Assumptions have been made here in order to give a simple proof but they could probably be extended, for example by continuity arguments. Particularly, the assumption that initial conditions are generic mainly spreads out problems in two dimensions, but it is clear that even in that case singular trajectories are not frequent, and they did not appear in numerical results of Section \ref{sec:num}.
\subsection{Number of contact points on a particular case}
We study now a particular case of Problem 1 in order to show that optimal trajectories meet the state constraint only a very limited number of times. More precisely, Corollary \ref{co:contacts} states that in the conditions of Lemma \ref{le:sing_ngen} there can be at most two contacts before reaching the final point. To complete the proof, we made the assumption that the mass $m$ is constant (i.e.\ $q=0$). Let us start with some definitions.

Given a trajectory, we say that $[t_{c_1},t_{c_2}]$ is a \textit{boundary interval} if $h(t)=0$ for all $t \in [t_{c_1},t_{c_2}]$, and $[t_{c_1},t_{c_2}]$ is the largest interval satisfying this condition and containing $t_{c_1},t_{c_2}$. When the boundary interval is reduced to a point $t_c$ (i.e.\ $t_{c_1}=t_{c_2}=t_c$), we rather say that $t_c$ is a \textit{contact point}.

\begin{lem}\label{le:contact}
There is at most one contact point or boundary interval on each Max or Min arc.
\end{lem}
\begin{proof}


Let us show that there is at most one contact point or boundary interval on each Max arc (including possibly the final point for the last Max arc). A similar reasoning on Min arcs will then give the conclusion. By contradiction, assume that the same Max arc contains two different boundary intervals $[t'_{c_1}, t_{c_1}]$ and $[t_{c_2}, t'_{c_2}]$, with $t_{c_1}< t_{c_2}$. We can moreover assume that $h(t)=z(t)>0$ on $(t_{c_1} , t_{c_2})$. Then
\begin{equation} 
\dot{v}_z (t_{c_1}) \geq 0, \quad \dot{v}_z(t_{c_2}) \geq 0,
\end{equation}
and there exists $t_b \in (t_{c_1} , t_{c_2})$ such that 
$\dot{v}_z (t_{b}) < 0$.
Note that $\dot v_z$ is an affine function of $d_z$, the vertical component of $d$,
\[ 
\dot v_z = u_{max}\frac{T}{m}d_z - g_0,
\]
so the above sign condition on $\dot v_z$ write as
\begin{equation}\label{eq:tb}
d_z(t_{c_1}) \hbox{ and } d_z(t_{c_2})\geq \frac{mg_0}{Tu_{max}}, \qquad d_z(t_{b}) < \frac{mg_0}{Tu_{max}}.
\end{equation}

Now, on $(t_{c_1},t_{c_2})$, the state constraint is inactive, therefore $q_r = q_r(t_{c_1})$ is constant and  
\[
p_v(t) = p_0 - q_rt, \quad \forall t \in  (t_{c_1},t_{c_2}).
\]
Assume first that $p_0 $ and $q_r$ are collinear, and write $p_0 = \rho_0 \delta$ and $q_r = \rho_r \delta$, with $\rho_0,\rho_r \in \R$ and $\delta \in \mathcal{S}^2$. Thus, 
    \[\frac{p_v}{\| p_v \|} =  \text{sign}(\rho_0 - \rho_r t)\delta.\]
    We deduce that in that case $d_z$ can take only two values, $\cos(\theta)$ and the  constant value $|\delta_z|$ (if this value belongs to $(\cos(\theta), 1]$), and can change value at most one time. This contradicts \eqref{eq:tb}.

Thus $p_0$ and $q_r$ are not collinear. In particular $d$ is absolutely continuous on $(t_{c_1},t_{c_2})$.
First, let us notice that when the pointing constraint is active, $d_z$ is constantly equal to $\cos(\theta)$, and study now the evolution of $d_z$ when the pointing constraint is not active. We will reduce the problem to two dimensions. Let us choose $\hat n = \pm \frac{p_0 \wedge q_r}{\|  p_0 \wedge q_r\|}$ such that $\sca{\hat n}{ e_z} \geq 0$. Then $d$, which is equal to $\frac{p_v}{\|p_v\|}$, belongs to the normal plane to $\hat n$, denoted $\hat n^\perp$. Note that $\hat n \neq e_z$. Indeed, otherwise $p_{v_z} = 0$, which implies that $d_z$ is constantly equal to $\cos(\theta)$ and is in contradiction with \eqref{eq:tb}. Let $\alpha$ be the angle between $\hat n$ and the plane $(e_x,e_y)$ and let us choose $(u_1,u_2)$ an orthonormal basis of $ \hat n ^\perp $ such that $\sca{u_1}{e_z} = 0$  and $\sca{u_2}{e_z} = \cos(\alpha)> 0$. Then, we define $\phi$ such that $d$ can be written 
 \[d = \cos(\phi) u_1 + \sin(\phi) u_2\]
        with $\phi \in [-\frac{\pi}{2}, \frac{\pi}{2}]$ and we have that \[d_z = \sin(\phi) \cos(\alpha).\]
         We have from \eqref{eq:tb} that $\phi(t_{c_1})$ and $\phi(t_{c_2})$ are in $(0, \frac{\pi}{2})$. Now, let us show that the evolution of $d_z$ contradicts $\eqref{eq:tb}$. Since $\alpha$ is constant, 
        \begin{equation} \label{eq:ddz}\dot d_z = \dot \phi \cos(\phi) \cos(\alpha),\end{equation}
        therefore, $\dot d_z$ has the same sign as $\dot \phi$, and $\dot \phi$ can be expressed thanks to the following computations. We reduce ourselves to $ \hat n^\perp$, since $p_v$, $p_0$ and $q_r$ belongs to it. We place ourselves in this plane in the coordinates defined by $(u_1, u_2)$. By abuse of notation, we will call $p_v$ the vector in two dimensions defined by $p_v = (\sca{p_v}{u_1}, \sca{p_v}{u_2})$.
        As $p_v = \|p_v\|d = \|p_v\| \begin{pmatrix} \cos(\phi)\\ \sin(\phi) \end{pmatrix}$,
        then
        \begin{align*}
            \dot p_v &= \frac{d\|p_v\|}{dt}\begin{pmatrix} \cos(\phi)\\ \sin(\phi) \end{pmatrix} + \|p_v\| \dot \phi \begin{pmatrix} -\sin(\phi)\\ \cos(\phi) \end{pmatrix}, \\
        \end{align*}
        and by multiplying on the left by $\begin{pmatrix} - \sin(\phi) & \cos(\phi) \end{pmatrix}$ we obtain
        \begin{align*}
             \begin{pmatrix} - \sin(\phi)\\ \cos(\phi) \end{pmatrix}^T (-q_r) = \|p_v\| \dot \phi.
        \end{align*}
        We deduce that
        \begin{align*}
             \dot \phi &= -\frac{1}{\|p_v\|^2}\det(q_r,p_v) = -\frac{1}{\|p_v\|^2}\det(q_r,p_0).\\
        \end{align*}
        As $q_r$ and $p_0$ are constant, we deduce that $\dot \phi$ is of constant sign, and $\dot d_z$ also from \eqref{eq:ddz}. Thus, $d_z$ is monotonous when the pointing constraint is not active and constant when it is active: we conclude that it is not possible to verify \eqref{eq:tb}.


\end{proof}
\begin{coro}\label{co:contacts}
For generic initial conditions, there is at most three contact points or boundary intervals along the trajectory.
More precisely,
\begin{enumerate}
    \item if $u_{min} < \frac{m_0 g_0}{T}$, there is along the trajectory at most two contact points or boundary intervals;
    \item if $u_{min} \cos(\theta) \geq \frac{m_0 g_0}{T}$, the only possible contact point is the final point. 
\end{enumerate}
\end{coro}
\begin{proof}
The main statement follows by application of Theorem \ref{th:BOB} and Lemma \ref{le:sing_ngen}, which imply that the control is of the form Max-Min-Max (or the restriction of such a control to a subinterval), and Lemma \ref{le:contact} gives us the number of contacts.
Then, the assumption of Point 1 implies that $u_{min}T$ does not compensate the weight of the vehicle, therefore it is not possible to have a contact on a Min arc without violating the state constraint. Finally, the assumption of Point 2 implies that $u(t)T$ compensate the weight of the vehicle for any control direction, therefore the vertical velocity, and so the altitude, would remain positive after a contact with the state constraint. We deduce that it is not possible to have a contact aside from the final point.
\end{proof}

\begin{rem}
The last result has been demonstrated with the assumption of a constant mass. However we believe that it stays true if the mass varies a little. Indeed, the reasoning is still correct if $m(t_b) \sim m(t_{c_1})$ or $m(t_{c_2})$, therefore it is sufficient to assume that the mass varies only slightly between two contact points. Likewise, the same reasoning would work with a glide-slope constraint of $\gamma \neq 0$, by assuming that $n$ is constant, which means in 2 dimensions that the trajectory stays in the same half-plane $x \geq 0$ or $x \leq 0
$.
\end{rem}

\section{Numerical results} \label{sec:num}
This section presents some examples for a Mars powered descent problem, in two dimensions for the sake of simplicity. The simulations are carried out by using CasADi (\cite{Andersson2019}) with python language and the IPOPT solver. They are performed under the same conditions as in \cite{Acikmese2007}. The launcher parameters are $T = 16573 N$, $u_{min} = 0.3$ and $u_{max} = 0.8$ and $m_e = 1505kg$ and $g_0= 3.71m/s^2$ corresponds to the Mars gravitational constant. The initial state is given by  $r_0 = [2000, 1500]m $, $ v_0 = [100, -75]m/s$ and $m_0 = 1905 kg$.
The optimal control problem considered will aim to perform a pinpoint landing by steering the vehicle to null final position and velocity.  

The first set of simulations is performed in the conditions of section IV with a vehicle mass constant, i.e. with $q=0$. In the usual way, we would minimize fuel consumption, as in the next example, with a cost proportional to \[ J = \int \limits_0^{t_f} \|u\| dt.\]
Here, we minimize this cost to obtain the solution of an equivalent problem. It is straightforward to show that all results presented in Section III and IV remain valid under these conditions. Note that the assumption of the Point $1$ of Corollary \ref{co:contacts} $u_{min} < \frac{m_0 g_0}{T}$  is verified.  

\def\scal{1}
\def\w{0.35\linewidth}
\def\h{1*\w}
\def\tickfont{\small}
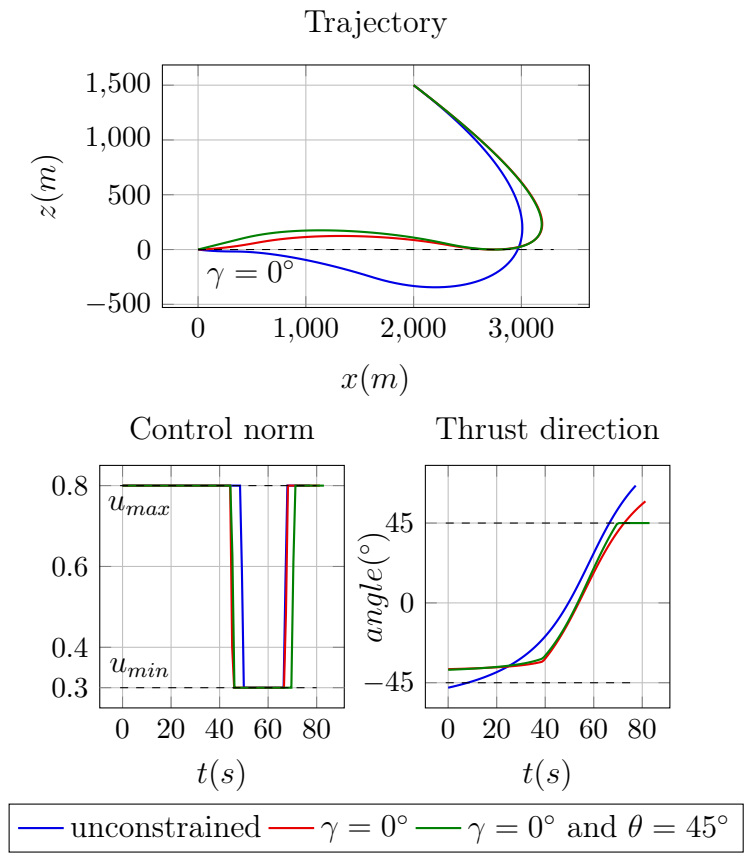
\begin{figure}
\centering
\begin{center}
\begin{tabular}{lr}
\multicolumn{2}{c}{
\begin{tikzpicture}[scale=\scal, baseline,trim axis left]
\begin{axis}[%
 	mark size = 1pt,
    width= 1.5*\w ,
    height = \h ,
	grid = major,
	xlabel = $x(m)$,
	ylabel = $z(m)$,
	title = {Trajectory},
	tick label style={font=\tickfont},
	legend columns= -1,
    legend entries = {unconstrained, $\gamma = 0^\circ$, $\gamma = 0^\circ$ and $\theta = 45^\circ$},
    legend to name =  {named},
]

\addplot[black!10!blue, thick] table [x index = {1}, y index = {2}, col sep = comma] {trajq0wc.csv};
\addplot[black!10!red, thick] table [x index = {1}, y index = {2}, col sep = comma] {trajq0ac.csv};
\addplot[black!50!green, thick] table [x index = {1}, y index = {2}, col sep = comma] {trajq0pac.csv};
\addplot[black, dashed, domain = -1:3300] {0.0} node[pos=0.15,below]{$\gamma = 0^\circ$};
\end{axis}
\end{tikzpicture}
}
\\
\hspace{0.1cm}
\begin{tikzpicture}[scale=\scal,baseline,trim axis left]
\begin{axis}[%
 	mark size = 1pt,
    width= \w ,
    height = \h ,
	grid = major,
	xlabel = $t(s)$,
	title = {Control norm},
	extra y ticks = {0.0, 0.3},
	tick label style={font=\tickfont},
]

\addplot[black!10!blue, thick] table [x index = {0}, y index = {3}, col sep = comma] {contq0wc.csv};
\addplot[black!10!red, thick] table [x index = {0}, y index = {3}, col sep = comma] {contq0ac.csv};
\addplot[black!50!green, thick] table [x index = {0}, y index = {3}, col sep = comma] {contq0pac.csv};
\addplot[black, dashed, domain = -1:80] {0.8}node[pos=0.1,below]{$u_{max}$};
\addplot[black, dashed, domain = -1:80] {0.3}node[pos=0.1,above]{$u_{min}$};
\end{axis}
\end{tikzpicture}
&
\hspace{0.5cm}
\begin{tikzpicture}[scale=\scal, baseline,trim axis left]
\begin{axis}[%
 	mark size = 1pt,
    width= \w ,
    height = \h ,
	grid = major,
	xlabel = $t(s)$,
	ylabel = $angle (^\circ)$,	
	title = {Thrust direction},
    ytick ={-45,0, 45},
    ylabel style={at={(-0.1,0.5)}},
	tick label style={font=\tickfont},
]

\addplot[black!10!blue, thick] table [x index= {0}, y expr = atan(\thisrowno{1}/\thisrowno{2}), col sep = comma] {contq0wc.csv};
\addplot[black!10!red, thick] table [x index= {0}, y expr = atan(\thisrowno{1}/\thisrowno{2}), col sep = comma] {contq0ac.csv};
\addplot[black!50!green, thick] table [x index= {0}, y expr = atan(\thisrowno{1}/\thisrowno{2}), col sep = comma] {contq0pac.csv};

\addplot[black, dashed, domain = -1:75] {45} node[pos = 1, below]{};
\addplot[black, dashed, domain = -1:75] {-45}node[pos = 0.9, above]{};
\end{axis}
\end{tikzpicture}
\end{tabular}
\ref*{named}
\end{center}     
\captionof{figure}{Simulation results for $q=0$ without glide-slope or pointing constraint (blue), with only a glide-slope constraint of $\gamma = 0^\circ$ (red) and with glide-slope and pointing constraints for $\theta = 45^\circ$ (green). }\label{fi:MarsLandq0}
\end{figure}

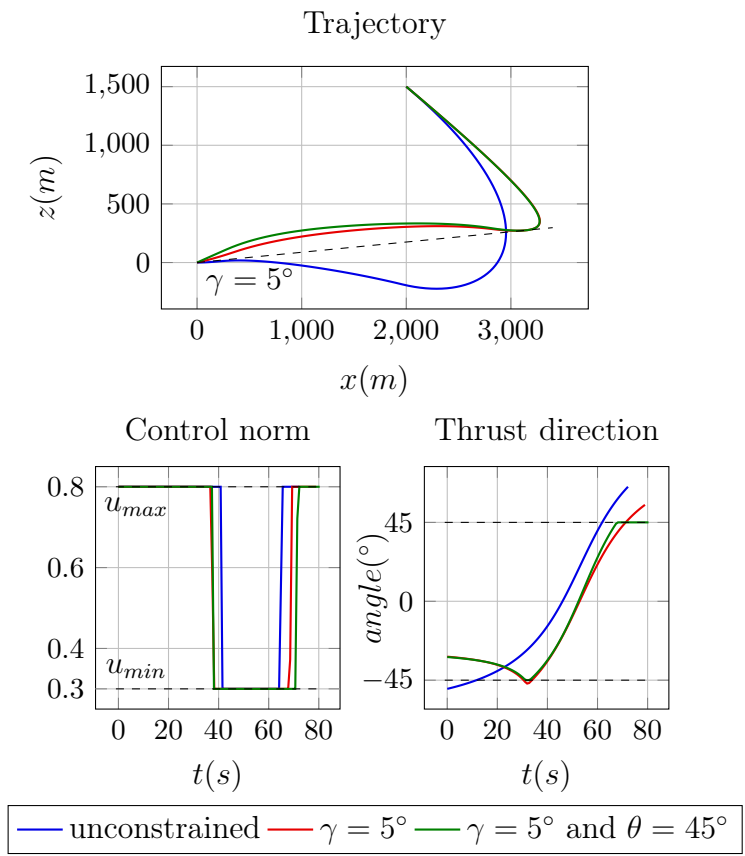
\begin{figure}
\centering
\begin{center}
\begin{tabular}{lr}
\multicolumn{2}{c}{
\begin{tikzpicture}[scale=\scal, baseline,trim axis left]
\begin{axis}[%
 	mark size = 1pt,
    width= 1.5*\w ,
    height = \h ,
	grid = major,
	xlabel = $x(m)$,
	ylabel = $z(m)$,
	title = {Trajectory},
	tick label style={font=\tickfont},
	legend columns= -1,
    legend entries = {unconstrained, $\gamma = 5^\circ$, $\gamma = 5^\circ$ and $\theta = 45^\circ$},
    legend to name =  {named},
]

\addplot[black!10!blue, thick] table [x index = {1}, y index = {2}, col sep = comma] {trajwc.csv};
\addplot[black!10!red, thick] table [x index = {1}, y index = {2}, col sep = comma] {trajgsc.csv};
\addplot[black!50!green, thick] table [x index = {1}, y index = {2}, col sep = comma] {trajpgsc.csv};
\addplot[black, dashed, domain = -1:3400] {tan(5)*x} node[pos=0.15,below]{$\gamma = 5^\circ$};
\end{axis}
\end{tikzpicture}
}
\\
\hspace{0.1cm}
\begin{tikzpicture}[scale=\scal,baseline,trim axis left]
\begin{axis}[%
 	mark size = 1pt,
    width= \w ,
    height = \h ,
	grid = major,
	xlabel = $t(s)$,
	title = {Control norm},
	extra y ticks = {0.0, 0.3},
	tick label style={font=\tickfont},
]

\addplot[black!10!blue, thick] table [x index = {0}, y index = {3}, col sep = comma] {contwc.csv};
\addplot[black!10!red, thick] table [x index = {0}, y index = {3}, col sep = comma] {contgsc.csv};
\addplot[black!50!green, thick] table [x index = {0}, y index = {3}, col sep = comma] {contpgsc.csv};
\addplot[black, dashed, domain = -1:80] {0.8}node[pos=0.1,below]{$u_{max}$};
\addplot[black, dashed, domain = -1:80] {0.3}node[pos=0.1,above]{$u_{min}$};
\end{axis}
\end{tikzpicture}
&
\hspace{0.5cm}
\begin{tikzpicture}[scale=\scal, baseline,trim axis left]
\begin{axis}[%
 	mark size = 1pt,
    width= \w ,
    height = \h ,
	grid = major,
	xlabel = $t(s)$,
	ylabel = $angle (^\circ)$,	
	title = {Thrust direction},
    ytick ={-45,0, 45},
    ylabel style={at={(-0.1,0.5)}},
	tick label style={font=\tickfont},
]
\addplot[black!10!blue, thick] table [x index= {0}, y expr = atan(\thisrowno{1}/\thisrowno{2}), col sep = comma] {contwc.csv};
\addplot[black!10!red, thick] table [x index= {0}, y expr = atan(\thisrowno{1}/\thisrowno{2}), col sep = comma] {contgsc.csv};
\addplot[black!50!green, thick] table [x index= {0}, y expr = atan(\thisrowno{1}/\thisrowno{2}), col sep = comma] {contpgsc.csv};
\addplot[black, dashed, domain = -1:80] {45} node[pos = 1, below]{};
\addplot[black, dashed, domain = -1:80] {-45}node[pos = 0.9, above]{};
\end{axis}
\end{tikzpicture}
\\
\end{tabular}
\ref*{named}
\end{center}     
\captionof{figure}{Simulation results with a varying mass, and a glide-slope constraint of $\gamma = 5^\circ$. }\label{fi:MarsLand}
\end{figure}
Three simulations were performed using the conditions described above. The first executes a Mars landing if no pointing and altitude constraint is applied. Figure \ref{fi:MarsLandq0} shows that in that case, the zero altitude is crossed. In the second simulation, an altitude constraint has been added. We see that, in accordance with the results of Corollary \ref{co:contacts}, there are two contact points with the state constraint along the trajectory: one during the first Bang arc and the final point. However, the angle of the thrust exceeds $45^\circ$ at the end of the trajectory. Finally, in the third example, a pointing constraint of $45^\circ$ is added. There are again two contacts with the state constraints, and the pointing constraint is active during the last $16s$ of the trajectory, which has the effect of making the trajectory more vertical. \par

A second set of simulations has been performed with a varying mass, with $q = 8.4294 kg/s$, considering a glide-slope constraint of $\gamma = 5^\circ$ and maximizing the final mass of the vehicle: \[ J =  - m(t_f) = - m(0) + \int \limits_0^{t_f} q\|u\| dt .\] In that case, the constraints are more compelling, as we see on the green plot on Figure \ref{fi:MarsLand} that the thrust direction is saturated by the pointing constraint on both its inferior and its superior bounds. There is still one contact point with the state constraint, although we leave the framework of Corollary \ref{co:contacts}. Remark that in all examples the form of the control is Max-Min-Max, and even the switching times varies little from one simulation to another.

\appendix

\section{Appendix}
\begin{lem}
\label{le:kkt}
The solutions of the maximization problem
\begin{equation}
\max \limits_d \sca{p_v}{d} \quad \text{under the conditions} \,
\begin{cases}
 c_+ (d) = \sca{e_z}{ d} - \cos(\theta) \geq 0, \\
 c_1 (d) = \sca{d}{d} -1 = 0,
\end{cases}
\end{equation}
are the vectors $d \in \mathcal{S}^2$ satisfying
\begin{equation}
d  =
\begin{cases}
\frac{p_v}{\Vert p_v \Vert} \; & \normalfont\text{if}\, p_{v_z} \geq \Vert p_v \Vert \cos(\theta) \ \text{and} \ p_v \neq 0,\\[1mm]
\left( \sin(\theta) \frac{\overline p_v}{\Vert \overline p_v \Vert},  \cos(\theta) \right) \,  & \normalfont\text{if}\, p_{v_z} < \Vert p_v \Vert \cos(\theta) \ \text{and} \ \overline p_v \neq 0,\\[1mm]
\Big(  \sin(\theta) \delta,  \cos(\theta) \Big) \ \normalfont \text{where}\ \delta \in \mathcal{S}^1 \ & \normalfont\text{if}\ p_{v_z} < \Vert p_v \Vert \cos(\theta) \ \text{and} \ \overline p_v = 0.
\end{cases}
\end{equation}
\end{lem}

\begin{proof}
Set $f = \sca{p_v}{d}$.
Denoting by $\lambda_+$ and $\lambda$ the Lagrange multipliers associated to the conditions $c_+$ and $c_1$, the Karush-Kuhn-Tucker optimality conditions write as
\begin{equation}
\begin{cases}
\nabla f(d) + \lambda \nabla c_1(d) + \lambda_+ \nabla c_+(d) =0\\
c_+ (d)\geq 0, \; c_1(d) = 0, \\
\lambda_+ c_+(d) = 0,\\
\lambda_+ \geq 0,
\end{cases}
\end{equation}
where the gradients are given by
\begin{equation}
\nabla f (d)= p_v, \qquad \nabla c_1(d) = 2d, \qquad \nabla c_+ (d)= e_z.
\end{equation}

\underline{Case 1: $c_+(d) >0$.} Then $\lambda_+ = 0$, $p_v + 2\lambda d =0$, $ \sca{d}{d} = 1$, and we get the following alternative. Either $p_v=0$ and then every $d \in \mathcal{S}^2$ is a maximum. Or $d = \epsilon \frac{p_v}{\Vert p_v \Vert}$ with $\epsilon = \pm 1$, and since $d$ maximizes $f$, then $\epsilon = 1$.

\underline{Case 2: $c_+(d) =0$.} Then $d_z=\cos \theta$, $p_v + 2\lambda d + \lambda_+ e_z =0$, $\sca{d}{d} =1$, and we get the following alternative. Either $p_v=p_{v_z} e_z$ (i.e.\ $\overline p_v=0$) with $p_{v_z}\leq 0$, and in this case, choosing $\lambda_+=-p_{v_z}$ we obtain that every $d$ of the form $\begin{pmatrix} \sin(\theta) \delta, & \cos(\theta)\end{pmatrix}$, with $\delta \in \mathcal{S}^1$, is maximum. Or $p_v + \lambda_+ e_z \neq 0$ and so $d = \epsilon\frac{p_v + \lambda_+ e_z}{\Vert p_v + \lambda_+ e_z\Vert }$ with $\epsilon = \pm 1$. In that case, $\epsilon (p_{v_z} + \lambda_+) = \Vert p_v + \lambda_+ e_z \Vert \cos(\theta)$ and
\begin{align*}
 d &= \begin{pmatrix}  \frac{ \epsilon p_{v_x}}{\Vert p_v + \lambda_+ e_z \Vert} \\ \frac{\epsilon p_{v_y}}{\Vert p_v + \lambda_+ e_z \Vert} \\ \frac{ \epsilon p_{v_z}}{\Vert p_v + \lambda_+ e_z \Vert} \end{pmatrix} = \cos(\theta) \begin{pmatrix} \frac{p_{v_x}}{(p_{v_z} + \lambda_+)} \\ \frac{p_{v_y}}{ (p_{v_z} + \lambda_+)} \\ 1 \end{pmatrix}.
  \end{align*}
As $d$ is unit, we have
\[
\Vert d \Vert ^2 = 1 = \cos(\theta) ^2 \left( 1 + \frac{p_{v_x}^2 + p_{v_y}^2}{( p_{v_z} + \lambda_+) ^2}
\right) = \cos(\theta) ^2 \left( 1 + \frac{\Vert \overline p_v \Vert^2}{( p_{v_z} + \lambda_+) ^2}
\right).
\]
We deduce that
\[ \sin(\theta) ^2 = \cos(\theta)^2 \frac{ \Vert \overline p_v \Vert^2}{(p_{v_z}+ \lambda_+)^2}, \qquad \hbox{so} \qquad  \frac{ \cos \theta}{(p_{v_z}+ \lambda_+)}=\epsilon \frac{\sin \theta}{\Vert \overline p_v \Vert}.
\]
Thus,
\[ d = \begin{pmatrix} \epsilon \sin(\theta) \frac{p_{v_x}}{\Vert \overline{p_v} \Vert} \\ \epsilon \sin(\theta) \frac{p_{v_y}}{\Vert \overline{p_v} \Vert} \\ \cos(\theta)\end{pmatrix},
\]
and the parameter $\epsilon$ is equal to $+1$ since $d$ maximizes $f$.
\end{proof}

\bibliographystyle{ieeetr}
\bibliography{biblioPaper}
\addtolength{\textheight}{-12cm}

\end{document}